\documentclass[12pt]{amsart}
\usepackage{amssymb,amsmath,amsthm,mathtools}
\usepackage{enumerate}
\usepackage{pgfplots}
\pgfplotsset{compat=1.15}
\usepackage{hyperref}

\newtheorem{theorem}{Theorem}[section]
\newtheorem{lemma}[theorem]{Lemma}
\newtheorem{proposition}[theorem]{Proposition}

\theoremstyle{definition}
\newtheorem{definition}[theorem]{Definition}
\newtheorem{assumption}[theorem]{Assumption}
\theoremstyle{remark}
\newtheorem{remark}[theorem]{Remark}

\newcommand{\bk}{\mathbf{k}}
\newcommand{\bx}{\mathbf{x}}
\newcommand{\avg}[1]{\left\langle #1\right\rangle}
\newcommand{\Rey}{\mathrm{Re}}
\newcommand{\etam}{\eta_{\mathrm{model}}}
\newcommand{\etaK}{\eta_{\mathrm{K}}}
\newcommand{\us}{u_{\mathrm{small}}}
\newcommand{\Hsp}{\mathcal{H}}
\newcommand{\Vsp}{\mathcal{V}}
\newcommand{\Dsp}{\mathcal{D}}

\begin{document}
\emergencystretch=3em

\title[Time Relaxation with Modified Lavrentiev Deconvolution]{Time Relaxation with Iterative Modified Lavrentiev Regularization}

\author{Ming Zhong}
\address{Department of Mathematics, University of Houston, Houston, TX 77204, USA}
\email{mzhong3@central.uh.edu}

\date{\today}

\begin{abstract}
We study a time relaxation regularization of the Navier--Stokes equations,
\[u_t+u\cdot\nabla u+\nabla p-\nu\Delta u+\tfrac{\chi}{\delta}(u-D_N\bar u)=f,\]
in which $\bar u$ is the differential filter of radius $\delta$ applied to $u$ and $D_N$ is the iterated modified Lavrentiev deconvolution operator with regularization parameter $\alpha$.  We first show that the relaxation term has a simple closed form: it is a filtered hyperviscosity, $\nu_N(-\Delta)^{N+1}G_{\sqrt\alpha\delta}^{N+1}u$ with $\nu_N=\chi\alpha^{N+1}\delta^{2N+1}$, where $G_{\sqrt\alpha\delta}$ is the differential filter of the smaller radius $\sqrt\alpha\,\delta$.  From this we obtain sharp consistency bounds, the existence of weak solutions with an energy inequality, and an explicit estimate showing that the energy in scales below any given length is $O(\delta/\chi)$ as $\chi\to\infty$.  Extending a result of Pakzad for van Cittert deconvolution, we bound the time-averaged energy dissipation rate by $\frac{U^3}{\ell}\big(1+\Rey^{-1}+\Rey_N^{-1}\big)$, where $\Rey_N=UL^{2N+1}/\nu_N$ is the relaxation Reynolds number.  We then develop a K41-type similarity theory.  The model's micro-scale solves an explicit scalar equation, which has no solution below a threshold value of $\chi$ and two solutions above it; the larger one is the cut-off scale.  Requiring the micro-scale to equal the filter radius gives the scaling $\chi\simeq U(\delta/L)^{1/3}(1+1/\alpha)^{N+1}$, for which the consistency error of the relaxation term is $O(\delta^{2N+4/3})$ uniformly in $\alpha$.
\end{abstract}

\maketitle

\section{Introduction}\label{sec:intro}

Direct numerical simulation of a three-dimensional turbulent flow at Reynolds number $\Rey$ requires roughly $N_{\mathrm{dof}}^{\mathrm{NSE}}\simeq\Rey^{9/4}$ mesh points per time step \cite{H69, P00}, which is out of reach in most applications.  The largest structures of the flow carry most of its energy and are responsible for most of its mixing and momentum transport, and they are usually what a simulation is meant to predict.  For this reason one uses turbulence models, large eddy simulation, or numerical regularizations that truncate the small scales \cite{BIL06, J04, S01}.

One of the simplest regularizations is \emph{time relaxation}: a linear, lower-order term $\chi u^*$ is added to the momentum equation, where $u^*$ is a generalized fluctuation of the velocity, that is, the part of $u$ that varies over lengths below a chosen scale $\delta$.  The term damps the unresolved scales exponentially in time and leaves the resolved scales almost untouched.  The idea goes back to the regularized Chapman--Enskog expansion of Rosenau \cite{R89} and Schochet and Tadmor \cite{ST92}.  It was developed for flow simulation by Stolz, Adams and Kleiser \cite{SA99, SAK01a, SAK01b, AS02}, who take $u^*=u-D_N\bar u$ with $\bar u$ a filtered velocity and $D_N$ the van Cittert approximate deconvolution operator.  Layton and Neda \cite{LN07} showed that this model has an energy cascade that is truncated at a micro-scale determined by $\chi$, and they derived the scaling of $\chi$ that places the micro-scale at the filter radius.  Numerical analysis and further developments can be found in \cite{ELN07, CL10, IMMS12, LR12, BNH17, P19}.

In this paper the van Cittert operator is replaced by the \emph{iterated modified Lavrentiev} (or modified iterated Tikhonov--Lavrentiev) deconvolution operator of \cite{Mays11, IMMS12, MZ18}.  Let $\Omega=(0,L)^3$, and impose periodic boundary conditions with zero mean:
\begin{equation}\label{eq:periodic}
  \phi(x+Le_j,t)=\phi(x,t),\qquad\int_\Omega\phi(x,t)\,dx=0,\qquad\text{for }\phi=u,\,p,\,f,\,u_0 .
\end{equation}
We use the differential filter \cite{G86}: for a given $L$-periodic $\phi$, its average $\bar\phi=G\phi$ is the $L$-periodic solution of
\begin{equation}\label{eq:filter}
  -\delta^2\Delta\bar\phi+\bar\phi=\phi\quad\text{in }\Omega .
\end{equation}
For a regularization parameter $0<\alpha\le1$ and an integer $N\ge0$, define
\begin{equation}\label{eq:DN}
  D_0:=\big((1-\alpha)G+\alpha I\big)^{-1},\qquad D_N:=\sum_{j=0}^{N}(I-D_0G)^jD_0 .
\end{equation}
The bounded operator $D_N$ is an approximate inverse of $G$: for smooth $\phi$, $\phi=D_N\bar\phi+O(\alpha^{N+1}\delta^{2N+2})$ \cite{MZ18}.  The model we study is to find an $L$-periodic, zero-mean velocity $u$ and pressure $p$ satisfying
\begin{equation}\label{eq:model}
\begin{aligned}
  u_t+u\cdot\nabla u+\nabla p-\nu\Delta u+\frac{\chi}{\delta}\big(u-D_N\bar u\big)&=f &&\text{in }\Omega\times(0,T),\\
  \nabla\cdot u&=0 &&\text{in }\Omega\times(0,T),\\
  u(x,0)&=u_0(x) &&\text{in }\Omega .
\end{aligned}
\end{equation}
The relaxation parameter $\chi>0$ has the units of a velocity, so $\chi/\delta$ is a rate.  To use the model one must choose $\chi$, and analytical guidance on how it should scale with $\delta$, $\alpha$, $N$ and the flow is the goal of this paper.  Our approach follows Layton and Neda \cite{LN07}.

\subsection{A structural observation and its consequences}
The starting point is an identity from \cite{MZ18}: $D_0G$ is again a differential filter, with the smaller radius $\sqrt\alpha\,\delta$.  Consequently (Lemma~\ref{lem:structure})
\begin{equation}\label{eq:intro-structure}
  \frac{\chi}{\delta}\big(u-D_N\bar u\big)=\frac{\chi}{\delta}\big(I-G_{\sqrt\alpha\delta}\big)^{N+1}u
  =\nu_N\,(-\Delta)^{N+1}G_{\sqrt\alpha\delta}^{\,N+1}u,
\end{equation}
where $\nu_N:=\chi\,\alpha^{N+1}\delta^{2N+1}$ and $G_\gamma=(I-\gamma^2\Delta)^{-1}$.  Two conclusions follow at once.

First, the relaxation term is a \emph{filtered hyperviscosity} of order $N+1$ with coefficient $\nu_N$.  On scales larger than $\sqrt\alpha\,\delta$ it acts like $\nu_N(-\Delta)^{N+1}u$; on smaller scales it saturates at the bounded damping rate $\chi/\delta$.

Second, \eqref{eq:model} is the same equation as the Stolz--Adams--Kleiser time relaxation model with van Cittert deconvolution of order $N$, filter radius $\sqrt\alpha\,\delta$, and relaxation rate $\chi/\delta$.  Results for that model, in particular those of \cite{LN07}, therefore apply to \eqref{eq:model} after a change of parameters.  The present formulation adds a second length: the user-chosen cut-off $\delta$ and the effective filter radius $\sqrt\alpha\,\delta$ are decoupled.  We keep all three parameters $(\delta,\alpha,\chi)$ explicit and give direct proofs, so that the paper can be read independently.

Beyond the change of parameters, the paper contains the following additions to the theory of \cite{LN07}, which by the equivalence above also apply to van Cittert time relaxation ($\alpha=1$).  The estimate for the small scales as $\chi\to\infty$ is made explicit in all parameters and holds for an arbitrary cut-off wave number (Theorem~\ref{thm:truncation}; compare \cite[Theorem~3.7]{LN07}).  The relaxation term is identified as a filtered hyperviscosity, with a quantitative hyperviscous limit (Proposition~\ref{prop:hyper}).  Finally, the micro-scale equation is solved completely (Proposition~\ref{prop:roots}).  This explains the ``under-resolved'' case, for which \cite{LN07} obtained a micro-scale that decreases as $\chi$ increases and left its interpretation open: that value is the smaller of two solutions, and it is not the cut-off scale.

One result is an extension of known work and not new in substance.  For the van Cittert model, Pakzad \cite[Theorem~3.1]{P19} proved that the time-averaged energy dissipation rate is at most $\big(2+\Rey^{-1}+R_N^{-1}\big)\frac{U^3}{L}$, where $R_N$ is the non-dimensional parameter of \cite[Definition~4.1]{LN07}.  Theorem~\ref{thm:dissipation} carries this over to \eqref{eq:model}.  The differences are minor: the bound is stated for weak solutions, the three terms keep their own length scales, and the leading constant is $1$.

\subsection{Summary of results}
\begin{enumerate}[(i)]
  \item \emph{Deconvolution} (Section~\ref{sec:prelim}).  We prove \eqref{eq:intro-structure}, the bounds $0\le\big((I-D_NG)u,u\big)\le\|u\|^2$, and the consistency estimate $\|u-D_N\bar u\|\le(\tfrac{s}{N+1})^s(\alpha\delta^2)^s\|(-\Delta)^su\|$ for $0\le s\le N+1$.  We also quantify the hyperviscous limit: for fixed $\nu_N$, the relaxation term differs from $\nu_N(-\Delta)^{N+1}u$ by at most $(N+1)\alpha\delta^2\nu_N\|(-\Delta)^{N+2}u\|$.
  \item \emph{Energy estimates} (Section~\ref{sec:energy}).  Weak solutions exist and satisfy an energy inequality (Theorem~\ref{thm:existence}).  The energy in all scales below $1/\kappa$, integrated over $(0,T)$, is at most $\big(1+\frac{1}{\alpha\delta^2\kappa^2}\big)^{N+1}\frac{\delta}{2\chi}C_0$, with $C_0$ depending only on the data (Theorem~\ref{thm:truncation}).  So fluctuations below $O(\delta)$ tend to zero in $L^2(\Omega\times(0,T))$ as $\chi\to\infty$.
  \item \emph{Energy dissipation} (Theorem~\ref{thm:dissipation}, after \cite{P19}).  For a steady smooth body force, the time-averaged total dissipation rate satisfies
  $\avg{\varepsilon_\nu+\varepsilon_{\mathrm{model}}}\le\frac{U^3}{\ell}\big(1+\Rey_f^{-1}+\Rey_{N,f}^{-1}\big)$.
  When the force acts at the scale $L$, $\Rey_{N,f}$ reduces to $\Rey_N=UL^{2N+1}/\nu_N$.
  \item \emph{Similarity theory} (Section~\ref{sec:similarity}).  Under K41-type assumptions the model's micro-scale $\etam$ solves
  \[
     \Big(\frac{\chi}{\delta}\Big)^3\etam^2\Big(\frac{\alpha\delta^2}{\etam^2+\alpha\delta^2}\Big)^{3N+3}=\frac{U^3}{L}.
  \]
  This equation has no solution if $\chi<\chi_{\min}\simeq U(\delta/L)^{1/3}\alpha^{-1/3}$, and two solutions if $\chi>\chi_{\min}$ (Proposition~\ref{prop:roots}).  The larger solution is the cut-off scale.  It is never smaller than $\sqrt{\alpha/(3N+2)}\,\delta$.
  \item \emph{Parameter selection.}  Forcing $\etam=\delta$ gives
  \begin{equation}\label{eq:intro-chi}
     \chi\simeq\frac{U}{L^{1/3}}\,\delta^{1/3}\Big(1+\frac1\alpha\Big)^{N+1}.
  \end{equation}
  For this choice the consistency error of the relaxation term is $O\big((1+\alpha)^{N+1}\delta^{2N+4/3}\big)$, the number of degrees of freedom is $N_{\mathrm{dof}}\simeq(L/\delta)^3$ independently of $\Rey$, and the bound in (ii) for scales below $\delta$ becomes independent of $\alpha$ and $N$.
\end{enumerate}
The results in (i)--(iii) are theorems.  The results in (iv)--(v) are phenomenological, like K41 theory itself; the assumptions are stated explicitly in Section~\ref{sec:similarity}.

\section{Preliminaries: the filter and the deconvolution operator}\label{sec:prelim}

\subsection{Notation}
Throughout, $\|\cdot\|$ and $(\cdot,\cdot)$ denote the norm and inner product of $L^2(\Omega)$, for scalar or vector fields.  Every $L$-periodic, zero-mean $u\in L^2(\Omega)$ has the Fourier expansion
\begin{equation*}
  u(\bx)=\sum_{\bk}\hat u(\bk)e^{i\bk\cdot\bx},\qquad \hat u(\bk)=\frac{1}{L^3}\int_\Omega u(\bx)e^{-i\bk\cdot\bx}\,d\bx,
\end{equation*}
where the sum runs over the wave vectors $\bk=\frac{2\pi}{L}\mathbf n$, $\mathbf n\in\mathbb Z^3\setminus\{0\}$.  We write $k=|\bk|$ and $k_1=2\pi/L$ for the smallest wave number.  Parseval's identity reads $\|u\|^2=L^3\sum_{\bk}|\hat u(\bk)|^2$.  For $s\ge0$ we write $\|(-\Delta)^su\|^2=L^3\sum_{\bk}k^{4s}|\hat u(\bk)|^2$ whenever this is finite.  We use the spaces
\begin{align*}
  \Dsp&:=\{v\in C^\infty(\mathbb R^3)^3:\ v\text{ is $L$-periodic},\ \textstyle\int_\Omega v=0,\ \nabla\cdot v=0\},\\
  \Hsp&:=\text{closure of }\Dsp\text{ in }L^2(\Omega)^3,\qquad \Vsp:=\text{closure of }\Dsp\text{ in }H^1(\Omega)^3,
\end{align*}
and the dual norm $\|f\|_{-1}:=\sup_{v\in\Vsp\setminus\{0\}}(f,v)/\|\nabla v\|$.  The Poincar\'e inequality $k_1\|v\|\le\|\nabla v\|$ holds on $\Vsp$.

For $\gamma>0$ let $G_\gamma:=(I-\gamma^2\Delta)^{-1}$ be the differential filter of radius $\gamma$ on periodic, zero-mean functions.  It is the Fourier multiplier with symbol $\hat G_\gamma(k)=(1+\gamma^2k^2)^{-1}$.  We write $G=G_\delta$, $\bar u=Gu$ and $A=G^{-1}=I-\delta^2\Delta$.

\subsection{Deconvolution}
Deconvolution is central both in image processing \cite{BB98} and in large eddy simulation \cite{BIL06, Geu97, LL03, LL05}.  The problem is: given $\bar u$, find a useful approximation of $u$, that is, solve $Gu=\bar u$ approximately.  The filter $G$ is symmetric, positive and compact, and its inverse $A$ is unbounded, so the problem is ill-posed.  The modified iterated Tikhonov--Lavrentiev method \cite{Mays11, IMMS12, MZ18} computes approximations $u_n$ by the fixed-point iteration
\begin{equation}\label{eq:iteration}
\begin{aligned}
  \big((1-\alpha)G+\alpha I\big)u_0&=\bar u,\\
  \big((1-\alpha)G+\alpha I\big)(u_n-u_{n-1})&=\bar u-Gu_{n-1},\qquad n=1,2,\dots,N.
\end{aligned}
\end{equation}
For $\alpha>0$ the operator $(1-\alpha)G+\alpha I$ is symmetric with spectrum in $[\min(\alpha,1),\max(\alpha,1)]$, so each step is well posed.
\begin{definition}\label{def:DN}
The $N$-th iterated modified Lavrentiev deconvolution operator $D_N$ is the map $\bar u\mapsto u_N$ defined by \eqref{eq:iteration}.  We write $H_N:=D_NG$, so that $H_Nu=D_N\bar u$.
\end{definition}
Eliminating the intermediate steps in \eqref{eq:iteration} gives $u_n=D_0\bar u+(I-D_0G)u_{n-1}$ with $D_0=((1-\alpha)G+\alpha I)^{-1}$, and hence the formula \eqref{eq:DN} for $D_N$.

\begin{lemma}[Structure of $D_N$ and $H_N$]\label{lem:structure}
Let $\alpha>0$, $N\ge0$ and $\gamma:=\sqrt\alpha\,\delta$.  Then
\begin{equation}\label{eq:D0G}
  D_0=(I-\delta^2\Delta)(I-\alpha\delta^2\Delta)^{-1},\qquad D_0G=G_\gamma,
\end{equation}
\begin{equation}\label{eq:HN}
  H_N=I-(I-G_\gamma)^{N+1},\qquad \hat H_N(k)=1-\Big(\frac{z^2}{1+z^2}\Big)^{N+1},\quad z=\sqrt\alpha\,\delta k,
\end{equation}
and
\begin{equation}\label{eq:hyper}
  \frac{\chi}{\delta}(I-H_N)=\nu_N\,(-\Delta)^{N+1}G_\gamma^{\,N+1},\qquad \nu_N=\chi\,\alpha^{N+1}\delta^{2N+1}.
\end{equation}
In particular, $H_N$ is the van Cittert approximate deconvolution of order $N$ of the filter $G_\gamma$ composed with $G_\gamma$: $H_N=\sum_{j=0}^N(I-G_\gamma)^jG_\gamma$.
\end{lemma}
\begin{proof}
Since $G=(I-\delta^2\Delta)^{-1}$, we have $(1-\alpha)G+\alpha I=G\big[(1-\alpha)I+\alpha(I-\delta^2\Delta)\big]=G(I-\alpha\delta^2\Delta)$.  Inverting gives \eqref{eq:D0G}.  All operators involved are Fourier multipliers and commute, so by \eqref{eq:DN}
\[
  H_N=\sum_{j=0}^N(I-G_\gamma)^jG_\gamma=I-(I-G_\gamma)^{N+1},
\]
a telescoping sum.  The symbol of $I-G_\gamma$ is $1-(1+z^2)^{-1}=z^2/(1+z^2)$, which gives the formula for $\hat H_N$.  Finally $I-G_\gamma=-\gamma^2\Delta\,G_\gamma$, so $(I-H_N)=\gamma^{2N+2}(-\Delta)^{N+1}G_\gamma^{N+1}$, and $\frac{\chi}{\delta}\gamma^{2N+2}=\chi\alpha^{N+1}\delta^{2N+1}$.
\end{proof}

\begin{remark}[Relation to van Cittert time relaxation]\label{rem:equivalence}
By Lemma~\ref{lem:structure}, model \eqref{eq:model} with parameters $(\delta,\alpha,\chi,N)$ coincides with the time relaxation model of \cite{SAK01a, LN07}, $u_t+u\cdot\nabla u+\nabla p-\nu\Delta u+\chi'(u-D^{\mathrm{vC}}_N\bar u^{\gamma})=f$, with van Cittert deconvolution of order $N$, differential filter of radius $\gamma=\sqrt\alpha\,\delta$, and relaxation rate $\chi'=\chi/\delta$.  For $\alpha=1$ we have $D_0=I$, and $D_N=\sum_{j=0}^N(I-G)^j$ is exactly the van Cittert operator for $G$.
\end{remark}

\begin{lemma}[Properties of $H_N$]\label{lem:HN}
Let $0<\alpha\le1$ and $N\ge0$.
\begin{enumerate}[(a)]
  \item $H_N$ and $I-H_N$ are bounded, self-adjoint Fourier multipliers on $L^2(\Omega)$.  They commute with differentiation, and they map $\Hsp$ into $\Hsp$.  Their symbols satisfy $0<\hat H_N(k)<1$.  Hence, for all zero-mean $u\in L^2(\Omega)$,
  \begin{equation}\label{eq:positivity}
    0\le(H_Nu,u)\le\|u\|^2,\qquad \|u-H_Nu\|^2\le\big(u-H_Nu,u\big)\le\|u\|^2 .
  \end{equation}
  \item $D_N$ is bounded on $L^2(\Omega)$, with $\|D_N\|\le(N+1)/\alpha$.
  \item \emph{Consistency.}  If $0\le s\le N+1$ and $\|(-\Delta)^su\|<\infty$, then, with $0^0:=1$,
  \begin{equation}\label{eq:consistency}
    \|u-D_N\bar u\|=\|u-H_Nu\|\le\Big(\frac{s}{N+1}\Big)^s(\alpha\delta^2)^s\|(-\Delta)^su\| .
  \end{equation}
  In particular $\|u-D_N\bar u\|\le\alpha^{N+1}\delta^{2N+2}\|\Delta^{N+1}u\|$.
  \item \emph{Coercivity on small scales.}  Let $\kappa>0$ and let $P_\kappa$ be the $L^2$-orthogonal projection onto the Fourier modes with $k\ge\kappa$.  Then
  \begin{equation}\label{eq:coercive}
    \big(u-H_Nu,u\big)\ge\Big(\frac{\alpha\delta^2\kappa^2}{1+\alpha\delta^2\kappa^2}\Big)^{N+1}\|P_\kappa u\|^2 .
  \end{equation}
\end{enumerate}
\end{lemma}
\begin{proof}
(a) By \eqref{eq:HN}, $1-\hat H_N(k)=r^{N+1}$ with $r=z^2/(1+z^2)\in(0,1)$ for $k>0$.  A real Fourier multiplier that depends only on $k$ is self-adjoint, commutes with derivatives, and preserves the divergence-free, zero-mean constraint.  Parseval's identity gives $(H_Nu,u)=L^3\sum_{\bk}\hat H_N(k)|\hat u(\bk)|^2$ and the same for $I-H_N$.  Since $0<1-\hat H_N<1$, we have $(1-\hat H_N)^2\le1-\hat H_N$, which gives \eqref{eq:positivity}.

(b) The symbol of $D_N=H_NA$ is $(1+\delta^2k^2)(1-r^{N+1})$.  Since $1-r^{N+1}\le(N+1)(1-r)=(N+1)/(1+\alpha\delta^2k^2)$ and $\alpha\le1$, the symbol is at most $(N+1)(1+\delta^2k^2)/(1+\alpha\delta^2k^2)\le(N+1)/\alpha$.

(c) Let $t=\alpha\delta^2k^2$, so that $1-\hat H_N=(t/(1+t))^{N+1}$.  We claim that $(t/(1+t))^{N+1}\le(s/(N+1))^st^s$ for $t\ge0$ and $0\le s\le N+1$.  Indeed, with $r=t/(1+t)$ and $a=N+1-s\ge0$, the quotient is $r^a(1-r)^s$, whose maximum over $[0,1]$ is $a^as^s/(a+s)^{a+s}\le(s/(a+s))^s$.  Parseval's identity gives \eqref{eq:consistency}.

(d) The function $k\mapsto1-\hat H_N(k)$ is increasing.  Hence $(u-H_Nu,u)\ge L^3\sum_{k\ge\kappa}(1-\hat H_N(k))|\hat u(\bk)|^2\ge(1-\hat H_N(\kappa))\|P_\kappa u\|^2$.
\end{proof}

\begin{figure}[t]
\centering
\begin{tikzpicture}
\begin{semilogxaxis}[width=0.82\textwidth,height=0.48\textwidth,
  xlabel={$z=\sqrt{\alpha}\,\delta k$}, ylabel={$\hat H_N$},
  xmin=0.1,xmax=100,ymin=0,ymax=1.05,domain=0.1:100,samples=200,
  legend pos=south west,legend cell align=left,legend style={font=\small},grid=major]
\addplot[thick,black]{1-(x^2/(1+x^2))^(1)}; \addlegendentry{$N=0$}
\addplot[thick,blue,dashed]{1-(x^2/(1+x^2))^(2)}; \addlegendentry{$N=1$}
\addplot[thick,red,dotted]{1-(x^2/(1+x^2))^(6)}; \addlegendentry{$N=5$}
\addplot[thick,green!50!black,dashdotted]{1-(x^2/(1+x^2))^(11)}; \addlegendentry{$N=10$}
\addplot[thick,orange]{1-(x^2/(1+x^2))^(101)}; \addlegendentry{$N=100$}
\end{semilogxaxis}
\end{tikzpicture}
\caption{The transfer function $\hat H_N=1-\big(z^2/(1+z^2)\big)^{N+1}$ of $H_N=D_NG$.  Large scales ($z\ll1$) are reproduced almost exactly; small scales are attenuated.  The half-height point moves to $z\approx\sqrt{(N+1)/\ln2}$ as $N$ grows.}
\label{fig:transfer}
\end{figure}

Figure~\ref{fig:transfer} shows the transfer function $\hat H_N$ for several $N$.  $H_Nu$ is very close to $u$ on the large scales and attenuates the small scales.  The transition is located at $z=O(1)$, that is, at the length $\sqrt\alpha\,\delta$ and not at $\delta$.  More precisely, $\hat H_N=\frac12$ when $z^2=1/(2^{1/(N+1)}-1)\approx(N+1)/\ln2$, so the transition moves slowly to smaller scales as $N$ increases.

\begin{proposition}[Hyperviscous limit]\label{prop:hyper}
Let $\nu_N=\chi\alpha^{N+1}\delta^{2N+1}$.  If $\|(-\Delta)^{N+2}u\|<\infty$, then
\begin{equation*}
  \Big\|\frac{\chi}{\delta}(u-H_Nu)-\nu_N(-\Delta)^{N+1}u\Big\|\le(N+1)\,\alpha\delta^2\,\nu_N\,\|(-\Delta)^{N+2}u\| .
\end{equation*}
In particular, if $\alpha\to0$ and $\chi\to\infty$ with $\nu_N$ and $\delta$ fixed, the relaxation term converges to the hyperviscosity $\nu_N(-\Delta)^{N+1}u$.
\end{proposition}
\begin{proof}
By \eqref{eq:hyper}, the symbol of the difference is $\nu_Nk^{2N+2}\big(1-(1+t)^{-(N+1)}\big)$ with $t=\alpha\delta^2k^2$, and $0\le1-(1+t)^{-(N+1)}\le(N+1)t$ by Bernoulli's inequality.
\end{proof}
The family \eqref{eq:model} therefore interpolates between the Stolz--Adams--Kleiser time relaxation model ($\alpha=1$) and a hyperviscosity regularization of order $N+1$ ($\alpha\to0$ with $\nu_N$ fixed).

\section{Energy estimates}\label{sec:energy}

Set
\begin{equation*}
  \mathcal R(v):=\frac{\chi}{\delta}\big(v-H_Nv,v\big)\ge0 .
\end{equation*}
\begin{definition}\label{def:weak}
Let $u_0\in\Hsp$ and $f\in L^2(0,T;\Vsp')$.  A \emph{weak solution} of \eqref{eq:model} on $[0,T]$ is a function $u\in L^\infty(0,T;\Hsp)\cap L^2(0,T;\Vsp)$, weakly continuous from $[0,T]$ into $\Hsp$, with $u(0)=u_0$, such that for all $\varphi\in\Dsp$ and $t\in[0,T]$
\begin{equation}\label{eq:weakform}
\begin{aligned}
  (u(t),\varphi)-(u_0,\varphi)&+\int_0^t\Big[-(u\cdot\nabla\varphi,u)+\nu(\nabla u,\nabla\varphi)\Big]ds\\
  &+\frac{\chi}{\delta}\int_0^t(u-H_Nu,\varphi)\,ds=\int_0^t(f,\varphi)\,ds,
\end{aligned}
\end{equation}
and such that the energy inequality
\begin{equation}\label{eq:energyineq}
  \frac12\|u(t)\|^2+\int_0^t\big[\nu\|\nabla u\|^2+\mathcal R(u)\big]ds\le\frac12\|u_0\|^2+\int_0^t(f,u)\,ds
\end{equation}
holds for all $t\in[0,T]$.  A weak solution on $[0,\infty)$ is a function whose restriction to $[0,T]$ is a weak solution for every $T>0$.
\end{definition}
The next result is \cite[Proposition~3.5]{LN07} in the present setting.
\begin{theorem}[Existence]\label{thm:existence}
Let $0<\alpha\le1$, $N\ge0$, $\chi\ge0$, $\nu>0$, $u_0\in\Hsp$ and $f\in L^2_{\mathrm{loc}}(0,\infty;\Vsp')$.  Then \eqref{eq:model} has a weak solution on $[0,\infty)$.  Strong solutions satisfy \eqref{eq:energyineq} with equality.
\end{theorem}
\begin{proof}
The proof is the Leray--Hopf construction for the Navier--Stokes equations \cite{G00, T01}, and we only indicate where the relaxation term enters.  Let $u_M$ be the Galerkin approximation in the span of the Fourier modes with $k\le M$.  Because $I-H_N$ is a Fourier multiplier, it commutes with the Galerkin projection, and testing the Galerkin equation with $u_M$ gives the energy equality
\[
  \tfrac12\|u_M(t)\|^2+\int_0^t\big[\nu\|\nabla u_M\|^2+\mathcal R(u_M)\big]ds=\tfrac12\|u_M(0)\|^2+\int_0^t(f,u_M)\,ds .
\]
Since $\mathcal R\ge0$ by \eqref{eq:positivity}, $u_M$ is bounded in $L^\infty(0,T;\Hsp)\cap L^2(0,T;\Vsp)$ uniformly in $M$.  The relaxation term is bounded in $L^\infty(0,T;\Hsp)$ by $\frac{\chi}{\delta}\|u_M\|$, so the time derivatives are bounded exactly as for the Navier--Stokes equations, and a subsequence converges weakly in $L^2(0,T;\Vsp)$ and strongly in $L^2(0,T;\Hsp)$.  The relaxation term is linear and bounded, so it passes to the limit in \eqref{eq:weakform}.  The functional $v\mapsto\int_0^t\mathcal R(v)\,ds$ is convex and continuous on $L^2(0,T;\Hsp)$, hence weakly lower semicontinuous, which gives \eqref{eq:energyineq}.  A diagonal argument gives a solution on $[0,\infty)$.
\end{proof}
The model's kinetic energy is the same as for the Navier--Stokes equations, and its energy dissipation has two parts.  Per unit volume,
\begin{equation}\label{eq:epsdef}
\begin{gathered}
  E(u)(t):=\frac{1}{2L^3}\|u(t)\|^2,\qquad
  \varepsilon_\nu(u)(t):=\frac{\nu}{L^3}\|\nabla u(t)\|^2,\\
  \varepsilon_{\mathrm{model}}(u)(t):=\frac{1}{L^3}\,\frac{\chi}{\delta}\big(u-H_Nu,u\big).
\end{gathered}
\end{equation}
By Lemma~\ref{lem:HN}(a), $\varepsilon_{\mathrm{model}}\ge0$: the relaxation term always extracts energy.

\begin{theorem}[Truncation of small scales as $\chi\to\infty$]\label{thm:truncation}
Let $u$ be a weak solution of \eqref{eq:model} on $[0,T]$ and set
\begin{equation*}
  C_0:=\|u_0\|^2+\frac1\nu\int_0^T\|f\|_{-1}^2\,dt .
\end{equation*}
Then
\begin{equation}\label{eq:apriori}
  \sup_{0\le t\le T}\|u(t)\|^2+\nu\int_0^T\|\nabla u\|^2dt+\frac{2\chi}{\delta}\int_0^T\big(u-H_Nu,u\big)dt\le C_0 .
\end{equation}
Consequently,
\begin{equation}\label{eq:fluct}
  \int_0^T\|u-D_N\bar u\|^2dt\le\frac{\delta}{2\chi}\,C_0,
\end{equation}
and for every $\kappa>0$ the part of $u$ in the scales below $1/\kappa$ satisfies
\begin{equation}\label{eq:smallscales}
  \int_0^T\|P_\kappa u\|^2dt\le\Big(1+\frac{1}{\alpha\delta^2\kappa^2}\Big)^{N+1}\frac{\delta}{2\chi}\,C_0 .
\end{equation}
In particular, for fixed $\delta$, $\alpha$, $N$ and data, $u-D_N\bar u\to0$ and $P_\kappa u\to0$ in $L^2(\Omega\times(0,T))$ as $\chi\to\infty$.
\end{theorem}
\begin{proof}
In \eqref{eq:energyineq} use $(f,u)\le\|f\|_{-1}\|\nabla u\|\le\frac{1}{2\nu}\|f\|_{-1}^2+\frac\nu2\|\nabla u\|^2$ and multiply by $2$.  This gives \eqref{eq:apriori}.  Then \eqref{eq:fluct} follows from \eqref{eq:positivity}, and \eqref{eq:smallscales} from \eqref{eq:coercive}.
\end{proof}
\begin{remark}\label{rem:scaling-consistency}
Taking $\kappa=1/\delta$ in \eqref{eq:smallscales}, the energy in scales below $\delta$ is controlled by $(1+\frac1\alpha)^{N+1}\frac{\delta}{2\chi}C_0$.  The factor $(1+\frac1\alpha)^{N+1}$ is exactly the factor in the scaling \eqref{eq:intro-chi} that the similarity theory of Section~\ref{sec:similarity} predicts.  With that choice of $\chi$, the bound becomes $\frac{1}{2U}\delta^{2/3}L^{1/3}C_0$, independent of $\alpha$ and $N$.
\end{remark}
Theorem~\ref{thm:truncation} supports time relaxation as a general strategy, but it does not show how the scales are truncated or how $\chi$ should be chosen for a given $\delta$.  The next result takes a step in that direction.  It is the analogue for \eqref{eq:model} of Pakzad's bound for van Cittert time relaxation \cite[Theorem~3.1]{P19}.  Let
\begin{equation}\label{eq:timeavg}
  \avg{g}:=\limsup_{T\to\infty}\frac1T\int_0^Tg(t)\,dt
\end{equation}
denote the long-time average \cite{R95}.
\begin{theorem}[Time-averaged energy dissipation]\label{thm:dissipation}
Let $f\in\Dsp$ be independent of time, $f\ne0$, and let $u$ be a weak solution of \eqref{eq:model} on $[0,\infty)$.  Define
\begin{equation*}
  U:=\avg{\tfrac{1}{L^3}\|u\|^2}^{1/2},\qquad F:=\frac{\|f\|}{L^{3/2}},
\end{equation*}
and the length scales of the force
\begin{equation*}
  \ell:=\frac{F}{\|\nabla f\|_{L^\infty}},\quad \ell_1:=\Big(\frac{\|f\|}{\|\Delta f\|}\Big)^{1/2},\quad \ell_N:=\Big(\frac{\|f\|}{\|\Delta^{N+1}f\|}\Big)^{1/(2N+2)} .
\end{equation*}
Then $U<\infty$ and
\begin{equation}\label{eq:dissipation}
  \avg{\varepsilon_\nu(u)+\varepsilon_{\mathrm{model}}(u)}\le\frac{U^3}{\ell}+\frac{\nu U^2}{\ell_1^2}+\frac{\nu_NU^2}{\ell_N^{2N+2}}
  =\frac{U^3}{\ell}\Big(1+\frac{1}{\Rey_f}+\frac{1}{\Rey_{N,f}}\Big),
\end{equation}
where $\nu_N=\chi\alpha^{N+1}\delta^{2N+1}$ and
\begin{equation*}
  \Rey_f:=\frac{U\ell_1^2}{\nu\,\ell},\qquad \Rey_{N,f}:=\frac{U\ell_N^{2N+2}}{\nu_N\,\ell}=\frac{U\,\ell_N^{2N+2}}{\chi\,\alpha^{N+1}\delta^{2N+1}\,\ell} .
\end{equation*}
\end{theorem}
\begin{proof}
We follow the argument of Doering and Foias \cite{DF02} as adapted to time relaxation by Pakzad \cite{P19}.  We avoid Young's inequality in Step~3, which is why the leading constant is $1$ and not $2$.

\emph{Step 1: $U<\infty$.}  By the Poincar\'e inequality, $(f,u)\le\|f\|\|u\|\le\frac{1}{2\nu k_1^2}\|f\|^2+\frac\nu2\|\nabla u\|^2$.  Inserting this in \eqref{eq:energyineq} gives $\frac\nu2\int_0^T\|\nabla u\|^2\le\frac12\|u_0\|^2+\frac{T}{2\nu k_1^2}\|f\|^2$, and therefore
\begin{equation*}
  \frac1T\int_0^T\|u\|^2dt\le\frac{1}{k_1^2T}\int_0^T\|\nabla u\|^2dt\le\frac{\|u_0\|^2}{\nu k_1^2T}+\frac{\|f\|^2}{\nu^2k_1^4}.
\end{equation*}
So $U<\infty$.  Moreover, \eqref{eq:energyineq} and the Cauchy--Schwarz inequality give $\|u(T)\|^2\le\|u_0\|^2+2\|f\|\sqrt T\,(\int_0^T\|u\|^2)^{1/2}=O(T)$, hence
\begin{equation}\label{eq:boundary-term}
  \frac1T\big|(u(T)-u_0,f)\big|\to0\qquad\text{as }T\to\infty .
\end{equation}

\emph{Step 2: dissipation is bounded by power input.}  Divide \eqref{eq:energyineq} by $TL^3$ and use $\frac1T\int_0^T(f,u)\le\|f\|(\frac1T\int_0^T\|u\|^2)^{1/2}$.  Taking the limit superior gives
\begin{equation}\label{eq:FU}
  \avg{\varepsilon_\nu+\varepsilon_{\mathrm{model}}}\le FU .
\end{equation}

\emph{Step 3: a bound for $F$.}  Take $\varphi=f$ in \eqref{eq:weakform}:
\begin{equation*}
\begin{aligned}
  T\|f\|^2=(u(T)-u_0,f)&-\int_0^T(u\cdot\nabla f,u)\,dt-\nu\int_0^T(u,\Delta f)\,dt\\
  &+\frac{\chi}{\delta}\int_0^T\big(u,(I-H_N)f\big)dt .
\end{aligned}
\end{equation*}
Here we integrated by parts in the viscous term and used the self-adjointness of $I-H_N$.  Now $|(u\cdot\nabla f,u)|\le\|\nabla f\|_{L^\infty}\|u\|^2$, and by \eqref{eq:consistency} with $s=N+1$, $\frac{\chi}{\delta}\|(I-H_N)f\|\le\nu_N\|\Delta^{N+1}f\|$.  Divide by $T$, use $\frac1T\int_0^T\|u\|\le(\frac1T\int_0^T\|u\|^2)^{1/2}$ and \eqref{eq:boundary-term}, and take the limit superior:
\begin{equation*}
  F^2L^3\le\|\nabla f\|_{L^\infty}U^2L^3+\big(\nu\|\Delta f\|+\nu_N\|\Delta^{N+1}f\|\big)UL^{3/2}.
\end{equation*}
Dividing by $FL^3=\|f\|L^{3/2}$ and using the definitions of $\ell$, $\ell_1$, $\ell_N$,
\begin{equation*}
  F\le\frac{U^2}{\ell}+\frac{\nu U}{\ell_1^2}+\frac{\nu_NU}{\ell_N^{2N+2}} .
\end{equation*}
Multiplying by $U$ and combining with \eqref{eq:FU} proves \eqref{eq:dissipation}.
\end{proof}
\begin{remark}\label{rem:ReN}
If the force acts on a single wave number $k_f$, then $\ell_1=\ell_N=1/k_f$.  If all the length scales of the force are comparable to $L$, then
\begin{equation*}
  \Rey_{N,f}\simeq\frac{UL^{2N+1}}{\chi\,\alpha^{N+1}\delta^{2N+1}}=:\Rey_N .
\end{equation*}
This is the non-dimensional relaxation parameter that the similarity theory below obtains by a scaling argument (Definition~\ref{def:ReN}).  As observed in \cite{P19} for the van Cittert model, it enters a rigorous estimate in the same way that $\Rey$ does.  It also shows that the total energy dissipation rate is at most of order $U^3/\ell$ when $\Rey_f$ and $\Rey_{N,f}$ are large, which is the scaling of the energy input assumed in K41 theory.
\end{remark}

\section{A similarity theory of time relaxation with iterated modified Lavrentiev deconvolution}\label{sec:similarity}

We now consider \eqref{eq:model} at a Reynolds number so high, and a relaxation parameter so large, that viscous dissipation is negligible compared with the dissipation caused by time relaxation.  We ask whether the relaxation term truncates the scales of the solution, and at which length.  The arguments of this section are phenomenological, in the spirit of K41 theory \cite{F95, P00} and of the similarity theories of \cite{M96, LN07}.  The symbol $\simeq$ means equality up to a constant of order one.

\subsection{The relaxation rate and the non-dimensional parameters}
By Lemma~\ref{lem:structure}, the relaxation term damps a velocity component of length scale $\ell$ (wave number $1/\ell$) at the rate
\begin{equation}\label{eq:rate}
  \frac{\chi}{\delta}\,m(\ell),\quad m(\ell):=\Big(\frac{\alpha\delta^2}{\ell^2+\alpha\delta^2}\Big)^{N+1}=\Big(\frac{\alpha\delta^2}{\ell^2}\Big)^{N+1}\Big(1+\frac{\alpha\delta^2}{\ell^2}\Big)^{-(N+1)} .
\end{equation}
The Reynolds number of the Navier--Stokes equations is the ratio of the nonlinear term to the viscous term on the largest scales, $\Rey\simeq|u\cdot\nabla u|/|\nu\Delta u|\simeq(U^2/L)/(\nu U/L^2)=UL/\nu$.  The analogous quantity for the relaxation term is
\begin{equation}\label{eq:ReNratio}
\begin{aligned}
  \frac{|u\cdot\nabla u|}{\big|\frac{\chi}{\delta}(u-H_Nu)\big|}\simeq\frac{U^2/L}{\frac{\chi}{\delta}m(L)\,U}
  &=\frac{UL^{2N+1}}{\chi\alpha^{N+1}\delta^{2N+1}}\Big(1+\frac{\alpha\delta^2}{L^2}\Big)^{N+1}\\
  &\simeq\frac{UL^{2N+1}}{\chi\alpha^{N+1}\delta^{2N+1}},
\end{aligned}
\end{equation}
because $\delta\ll L$ and $0<\alpha\le1$.
\begin{definition}\label{def:ReN}
The non-dimensional time relaxation parameter of \eqref{eq:model} is
\begin{equation*}
  \Rey_N:=\frac{UL^{2N+1}}{\chi\,\alpha^{N+1}\delta^{2N+1}}=\frac{UL^{2N+1}}{\nu_N},\qquad N=0,1,2,\dots
\end{equation*}
\end{definition}
By \eqref{eq:hyper}, $\Rey_N$ is the Reynolds number of the hyperviscosity $\nu_N(-\Delta)^{N+1}$, and by Remark~\ref{rem:ReN} it is the quantity that appears in the dissipation bound \eqref{eq:dissipation}.

Next we form the corresponding parameter on the smallest persistent eddies.  Let $\us$ be a characteristic velocity of those eddies and $\etam$ their length scale.  Proceeding as in \eqref{eq:ReNratio}, but without assuming any order relation between $\etam$ and $\sqrt\alpha\,\delta$,
\begin{definition}\label{def:ReNsmall}
The non-dimensional parameter associated with the smallest persistent scales of \eqref{eq:model} is
\begin{equation*}
  \Rey_{N,\mathrm{small}}:=\frac{\us^2/\etam}{\frac{\chi}{\delta}m(\etam)\,\us}
  =\frac{\us\,\etam^{2N+1}}{\chi\,\alpha^{N+1}\delta^{2N+1}}\Big(1+\frac{\alpha\delta^2}{\etam^2}\Big)^{N+1}.
\end{equation*}
\end{definition}

\subsection{The micro-scale equation}
We make the following assumptions.
\begin{assumption}\label{ass:K41}
\begin{enumerate}[(a)]
  \item Energy is supplied at the large scales at the rate $\varepsilon\simeq U^3/L$ (compare Remark~\ref{rem:ReN}), and the flow is in statistical equilibrium: energy input at the large scales equals dissipation at the small scales.
  \item Viscous dissipation is negligible compared with the dissipation caused by time relaxation.
  \item The energy cascade is halted at the scale $\etam$ at which nonlinearity and relaxation balance: $\Rey_{N,\mathrm{small}}\simeq1$.
\end{enumerate}
\end{assumption}
The dissipation at the smallest persistent scales is, by \eqref{eq:epsdef} and \eqref{eq:rate}, $\varepsilon_{\mathrm{model}}\simeq\frac{\chi}{\delta}m(\etam)\us^2$.  Assumption~\ref{ass:K41} therefore gives the two equations
\begin{equation}\label{eq:pair}
  \frac{\us}{\etam}\simeq\frac{\chi}{\delta}\,m(\etam),\qquad\frac{U^3}{L}\simeq\frac{\chi}{\delta}\,m(\etam)\,\us^2 .
\end{equation}
Solving the first for $\us$ and inserting it in the second gives the equation that determines the model's micro-scale:
\begin{equation}\label{eq:microscale}
  \frac{U^3}{L}\simeq\Big(\frac{\chi}{\delta}\Big)^3\etam^2\,m(\etam)^3
  =\Big(\frac{\chi}{\delta}\Big)^3\etam^2\Big(\frac{\alpha\delta^2}{\etam^2+\alpha\delta^2}\Big)^{3N+3}.
\end{equation}
Equation \eqref{eq:microscale} has a simple interpretation.  In an inertial range with energy flux $\varepsilon$, an eddy of size $\ell$ has velocity $(\varepsilon\ell)^{1/3}$ and turnover rate $\varepsilon^{1/3}\ell^{-2/3}$.  The ratio of the relaxation rate to the turnover rate is
\begin{equation}\label{eq:rho}
  \rho(\ell):=\frac{\frac{\chi}{\delta}m(\ell)}{\varepsilon^{1/3}\ell^{-2/3}},\qquad\rho(\ell)^3=\frac{\chi^3}{\varepsilon\,\delta^3}\,\ell^2m(\ell)^3,
\end{equation}
and \eqref{eq:microscale} states that $\rho(\etam)=1$.

The next proposition is elementary, but it changes the picture.  We state it with equality in \eqref{eq:microscale}.
\begin{proposition}[Solutions of the micro-scale equation]\label{prop:roots}
Let $\gamma=\sqrt\alpha\,\delta$, $\varepsilon=U^3/L$, and
\begin{equation*}
  \Phi(\eta):=\eta^2\Big(\frac{\gamma^2}{\eta^2+\gamma^2}\Big)^{3N+3},\qquad\eta>0,
\end{equation*}
so that \eqref{eq:microscale} reads $\Phi(\etam)=\varepsilon\delta^3/\chi^3$.
\begin{enumerate}[(a)]
  \item $\Phi$ is increasing on $(0,\eta_*)$ and decreasing on $(\eta_*,\infty)$, where
  \begin{equation*}
    \eta_*=\frac{\gamma}{\sqrt{3N+2}}=\sqrt{\frac{\alpha}{3N+2}}\;\delta,\qquad\Phi(\eta_*)=\frac{\gamma^2}{3N+2}\Big(\frac{3N+2}{3N+3}\Big)^{3N+3}.
  \end{equation*}
  \item Let
  \begin{equation}\label{eq:chimin}
    \chi_{\min}:=c_N\,\frac{U}{L^{1/3}}\Big(\frac{\delta}{\alpha}\Big)^{1/3},\qquad c_N:=(3N+2)^{1/3}\Big(\frac{3N+3}{3N+2}\Big)^{N+1}.
  \end{equation}
  If $\chi<\chi_{\min}$, equation \eqref{eq:microscale} has no solution, and $\rho(\ell)<1$ for all $\ell>0$.  If $\chi=\chi_{\min}$, the only solution is $\eta_*$.  If $\chi>\chi_{\min}$, there are exactly two solutions $\eta_-<\eta_*<\eta_+$, and $\rho(\ell)>1$ exactly for $\eta_-<\ell<\eta_+$.
  \item For $\chi>\chi_{\min}$, the solution $\eta_+$ is increasing in $\chi$, $\eta_-$ is decreasing in $\chi$, and
  \begin{equation}\label{eq:rootbounds}
    \eta_-\ge\Big(\frac{U^3\delta^3}{\chi^3L}\Big)^{1/2},\qquad
    \eta_+\le\Big(\frac{\chi^3L}{U^3}\Big)^{\frac{1}{6N+4}}\alpha^{\frac12+\frac{1}{6N+4}}\,\delta^{1-\frac{1}{6N+4}} .
  \end{equation}
  The first bound is asymptotically sharp when $\eta_-\ll\gamma$, and the second when $\eta_+\gg\gamma$.
  \item Conversely, for every $\eta\ge\eta_*$ there is exactly one $\chi\ge\chi_{\min}$ with $\eta_+=\eta$, namely
  \begin{equation}\label{eq:chiofeta}
    \chi(\eta)=\frac{U}{L^{1/3}}\,\frac{\delta}{\eta^{2/3}}\Big(1+\frac{\eta^2}{\alpha\delta^2}\Big)^{N+1}.
  \end{equation}
\end{enumerate}
\end{proposition}
\begin{proof}
(a) We have $\frac{d}{d\eta}\log\Phi=\frac2\eta-\frac{(6N+6)\eta}{\eta^2+\gamma^2}$, which is positive if and only if $\eta^2+\gamma^2>(3N+3)\eta^2$, that is, $\eta<\eta_*$.  At $\eta_*$, $\gamma^2/(\eta_*^2+\gamma^2)=(3N+2)/(3N+3)$.

(b) Since $\Phi(0^+)=\Phi(\infty)=0$, the equation $\Phi(\eta)=\varepsilon\delta^3/\chi^3$ has no, one or two solutions according to whether $\varepsilon\delta^3/\chi^3$ is larger than, equal to or smaller than $\Phi(\eta_*)$.  The condition $\varepsilon\delta^3/\chi^3\le\Phi(\eta_*)$ is equivalent to $\chi^3\ge\frac{U^3\delta}{L\alpha}(3N+2)\big(\frac{3N+3}{3N+2}\big)^{3N+3}=\chi_{\min}^3$.  The statements about $\rho$ follow from $\rho^3=\chi^3\Phi/(\varepsilon\delta^3)$.

(c) As $\chi$ increases, the level $\varepsilon\delta^3/\chi^3$ decreases; since $\Phi$ is increasing at $\eta_-$ and decreasing at $\eta_+$, $\eta_-$ decreases and $\eta_+$ increases.  The bounds follow from $\Phi(\eta)\le\eta^2$ and $\Phi(\eta)\le\gamma^{6N+6}\eta^{-(6N+4)}$, and these inequalities become equalities to leading order when $\eta\ll\gamma$ and $\eta\gg\gamma$ respectively.

(d) Solve \eqref{eq:microscale} for $\chi$; the function $\chi(\eta)=\delta\,(\varepsilon/\Phi(\eta))^{1/3}$ is increasing on $[\eta_*,\infty)$ by (a).
\end{proof}

Proposition~\ref{prop:roots} has the following consequences for the model.
\begin{enumerate}[(1)]
  \item \emph{A threshold.}  Time relaxation is a zeroth-order term: its damping rate never exceeds $\chi/\delta$, while the turnover rate $\varepsilon^{1/3}\ell^{-2/3}$ of the eddies grows without bound as $\ell\to0$.  If $\chi<\chi_{\min}$, relaxation is slower than the nonlinear transfer at \emph{every} scale, and within this theory the relaxation term cannot halt the cascade.  Then Assumption~\ref{ass:K41}(b) fails and viscosity must end the cascade.
  \item \emph{The cut-off is the larger solution.}  If $\chi>\chi_{\min}$, the cascade, which proceeds from large to small scales, first meets a scale with $\rho=1$ at $\ell=\eta_+$.  We therefore identify
  \begin{equation*}
    \etam=\eta_+ .
  \end{equation*}
  The band $\eta_-<\ell<\eta_+$ is dominated by relaxation.  Scales below $\eta_-$ are again not controlled by the relaxation term; there, the viscous term must act.
  \item \emph{A lower limit for the micro-scale.}  Since $\eta_+\ge\eta_*$, the model's micro-scale cannot be smaller than $\sqrt{\alpha/(3N+2)}\,\delta$, whatever the value of $\chi$.
\end{enumerate}

\subsection{The three regimes}
We now relate \eqref{eq:microscale} to the filter radius $\delta$.  Three cases arise, according to whether $\etam$ is larger than, smaller than, or equal to $\delta$.

\medskip\noindent\textbf{Case 1: fully resolved, $\etam>\delta$.}  Since $0<\alpha\le1$, we have $\alpha\delta^2\ll\etam^2$ when $\etam\gg\sqrt\alpha\,\delta$, so $1+\alpha\delta^2/\etam^2\simeq1$ and \eqref{eq:microscale} reduces to $\frac{U^3}{L}\simeq(\frac{\chi}{\delta})^3\etam^2(\alpha\delta^2/\etam^2)^{3N+3}$.  Solving for $\etam$,
\begin{equation}\label{eq:case1}
  \etam\simeq\Big(\frac{\chi^3L}{U^3}\Big)^{\frac{1}{6N+4}}\alpha^{\frac12+\frac{1}{6N+4}}\,\delta^{1-\frac{1}{6N+4}} .
\end{equation}
This is the upper bound for $\eta_+$ in \eqref{eq:rootbounds}.  In terms of the hyperviscosity coefficient, \eqref{eq:case1} reads $\etam\simeq(\nu_N^3/\varepsilon)^{1/(6N+4)}$, which is the Kolmogorov scale of a fluid with hyperviscosity $\nu_N(-\Delta)^{N+1}$; for $N=0$ and $\nu_0=\nu$ it is the usual $\etaK=(\nu^3/\varepsilon)^{1/4}$.

\medskip\noindent\textbf{Case 2: under-resolved, $\etam<\delta$.}  Three situations can occur, according to the size of $\alpha$ relative to $(\etam/\delta)^2$.
\begin{itemize}
  \item If $\alpha\ll(\etam/\delta)^2$, then $\etam\gg\sqrt\alpha\,\delta$ and $\etam$ is again given by \eqref{eq:case1}.
  \item If $\alpha\gg(\etam/\delta)^2$, then $1+\alpha\delta^2/\etam^2\simeq\alpha\delta^2/\etam^2$, $m\simeq1$, and \eqref{eq:microscale} gives
  \begin{equation}\label{eq:case2}
    \eta\simeq\sqrt{\frac{U^3\delta^3}{\chi^3L}} .
  \end{equation}
  This is the analogue of \cite[eq.~(4.7)]{LN07}.  It decreases as $\chi$ increases, and its interpretation was left open in \cite{LN07}.  Proposition~\ref{prop:roots} settles the question: by \eqref{eq:rootbounds} this is the \emph{smaller} solution $\eta_-$.  It is not the scale at which the cascade is cut off.  It is the lower edge of the band of scales on which relaxation dominates, and it is the scale below which the relaxation term no longer controls the nonlinearity.
  \item If $\alpha=(\etam/\delta)^2$, that is, $\etam=\sqrt\alpha\,\delta$, then $1+\alpha\delta^2/\etam^2=2$.  By \eqref{eq:chiofeta}, the relaxation parameter that enforces this is
  \begin{equation}\label{eq:chigamma}
    \chi\simeq\frac{U}{L^{1/3}}\,2^{N+1}\Big(\frac{\delta}{\alpha}\Big)^{1/3}.
  \end{equation}
  With this choice the micro-scale equals the effective filter radius $\sqrt\alpha\,\delta$ of Lemma~\ref{lem:structure}.  By Remark~\ref{rem:equivalence}, \eqref{eq:chigamma} is the scaling \cite[eq.~(4.8)]{LN07} for a filter of radius $\sqrt\alpha\,\delta$: there the relaxation rate is $\frac{U}{L^{1/3}}2^{N+1}\gamma^{-2/3}$, and here $\chi/\delta=\frac{U}{L^{1/3}}2^{N+1}(\sqrt\alpha\,\delta)^{-2/3}$.  The consistency error of the relaxation term (see Lemma~\ref{lem:consistency} below) is then $O(\chi\alpha^{N+1}\delta^{2N+1})=O(\alpha^{N+\frac23}\delta^{2N+\frac43})$.
\end{itemize}

\medskip\noindent\textbf{Case 3: perfectly resolved, $\etam=\delta$.}  This is the most important case for practical computing.  Now $1+\alpha\delta^2/\etam^2=1+\alpha$ and \eqref{eq:microscale} becomes $\frac{U^3}{L}\simeq\frac{\chi^3}{\delta}\big(\frac{\alpha}{1+\alpha}\big)^{3N+3}$.  Solving for $\chi$,
\begin{equation}\label{eq:chi}
  \chi\simeq\frac{U}{L^{1/3}}\,\delta^{1/3}\Big(1+\frac1\alpha\Big)^{N+1}.
\end{equation}
Since $\delta\ge\eta_*$ for every $0<\alpha\le1$, the scale $\delta$ lies on the upper branch, so \eqref{eq:chi} is indeed the value $\chi(\delta)$ of \eqref{eq:chiofeta} and $\chi(\delta)\ge\chi_{\min}$.  Note that $\chi\to\infty$ as $\alpha\to0$, as the estimates of Section~\ref{sec:energy} require (Remark~\ref{rem:scaling-consistency}).

\begin{lemma}[Consistency error of the relaxation term]\label{lem:consistency}
If $\|\Delta^{N+1}u\|<\infty$, then
\begin{equation*}
  \Big\|\frac{\chi}{\delta}\big(u-D_N\bar u\big)\Big\|\le\chi\,\alpha^{N+1}\delta^{2N+1}\,\|\Delta^{N+1}u\| .
\end{equation*}
For $\chi=\chi(\eta)$ given by \eqref{eq:chiofeta}, the right-hand side equals $\frac{U}{L^{1/3}}\eta^{-2/3}\big(\eta^2+\alpha\delta^2\big)^{N+1}\|\Delta^{N+1}u\|$.  In particular, for the choice \eqref{eq:chi},
\begin{equation}\label{eq:conserr}
  \Big\|\frac{\chi}{\delta}\big(u-D_N\bar u\big)\Big\|\le\frac{U}{L^{1/3}}(1+\alpha)^{N+1}\delta^{2N+\frac43}\,\|\Delta^{N+1}u\| .
\end{equation}
\end{lemma}
\begin{proof}
The first bound is \eqref{eq:consistency} with $s=N+1$, multiplied by $\chi/\delta$.  The rest is substitution.
\end{proof}
For smooth flow fields and $0<\alpha\le1$, the consistency error \eqref{eq:conserr} is $O(\delta^{2N+4/3})$ uniformly in $\alpha$: it lies between $\frac{U}{L^{1/3}}\delta^{2N+4/3}\|\Delta^{N+1}u\|$ (as $\alpha\to0$) and $2^{N+1}$ times this value ($\alpha=1$).  It tends to zero as $\delta\to0$, so solutions of the model with the scaling \eqref{eq:chi} are consistent with the Navier--Stokes equations to high order.

\begin{remark}[The role of $\alpha$]\label{rem:alpha}
With $\etam=\delta$ fixed, decreasing $\alpha$ moves the effective filter radius $\sqrt\alpha\,\delta$ below the cut-off $\delta$.  This reduces the constant in the consistency error by a factor of up to $2^{N+1}$, and by Proposition~\ref{prop:hyper} it makes the relaxation term closer to a pure hyperviscosity.  The price is a larger relaxation rate: $\chi/\delta\simeq\frac{U}{L^{1/3}}\delta^{-2/3}(1+\frac1\alpha)^{N+1}$, which makes the term stiffer and restricts explicit time stepping to $\Delta t\lesssim\delta/\chi$.
\end{remark}

\subsection{Interpreting the assumption that viscous dissipation is negligible}
Assumption~\ref{ass:K41}(b) holds when the Kolmogorov micro-scale of the Navier--Stokes equations, $\etaK=\Rey^{-3/4}L$, is small compared with the micro-scale induced by the relaxation term.  According to K41 theory, viscous dissipation is negligible at scales above $\etaK$.  We give two ways to make ``large enough $\chi$'' precise.

\emph{First interpretation: $\etam\gg\etaK$.}  At the scale $\etam$, the relaxation rate equals the turnover rate $\varepsilon^{1/3}\etam^{-2/3}$ by \eqref{eq:pair}, and the viscous rate is $\nu\etam^{-2}$.  The relaxation rate dominates if and only if $\etam^{4/3}\gg\nu\varepsilon^{-1/3}$, that is, if and only if $\etam\gg(\nu^3/\varepsilon)^{1/4}=\etaK$.  By Proposition~\ref{prop:roots}, $\etam=\eta_+\ge\eta_*$ and $\eta_+$ is increasing in $\chi$.  Hence:
\begin{itemize}
  \item If $\etaK\le\eta_*=\sqrt{\alpha/(3N+2)}\,\delta$, which is the typical situation since $\delta\gg\etaK$, the condition $\etam>\etaK$ holds for every $\chi\ge\chi_{\min}$.  It places no constraint on $\chi$ beyond the threshold \eqref{eq:chimin}.
  \item If $\etaK>\eta_*$, the condition is $\chi>\chi(\etaK)$ with $\chi(\cdot)$ from \eqref{eq:chiofeta}.  When moreover $\etaK\gg\sqrt\alpha\,\delta$, formula \eqref{eq:case1} applies and the condition becomes
  \begin{equation}\label{eq:lower1}
    \chi>\big(\Rey^{-\frac34}L\big)^{2N+\frac43}\,\frac{U}{L^{1/3}}\,\alpha^{-(N+1)}\delta^{-(2N+1)} .
  \end{equation}
\end{itemize}

\emph{Second interpretation.}  One can instead require that, at the Kolmogorov scale itself, relaxation is stronger than viscosity: $\Rey_{\mathrm{small}}\gg\Rey_{N,\mathrm{small}}$ at $\etam=\etaK$, where $\Rey_{\mathrm{small}}=\us\etaK/\nu$.  Equivalently, $\frac{\chi}{\delta}m(\etaK)>\nu/\etaK^2$, which is the mild condition
\begin{equation}\label{eq:lower2}
  \chi>\nu\Big(\frac{\delta}{\etaK}\Big)^{-2N}\delta^{-1}\alpha^{-(N+1)}\Big(1+\frac{\alpha\delta^2}{\etaK^2}\Big)^{N+1}
  =\frac{\nu\,\delta}{\etaK^2}\Big(1+\frac{\etaK^2}{\alpha\delta^2}\Big)^{N+1}.
\end{equation}

\section{Conclusions and open problems}\label{sec:conclusion}

Time relaxation with iterated modified Lavrentiev deconvolution is a filtered hyperviscosity, $\nu_N(-\Delta)^{N+1}G_{\sqrt\alpha\delta}^{N+1}u$ with $\nu_N=\chi\alpha^{N+1}\delta^{2N+1}$.  Equivalently, it is van Cittert time relaxation for a differential filter of radius $\sqrt\alpha\,\delta$.  The relaxation term extracts energy, it does not dissipate appreciable energy on the resolved scales ($\|u-D_N\bar u\|=O(\alpha^{N+1}\delta^{2N+2})$ for smooth $u$), and it drives the scales below any fixed length to zero as $\chi\to\infty$, with the explicit rate of Theorem~\ref{thm:truncation}.  The time-averaged dissipation rate is bounded by $\frac{U^3}{\ell}(1+\Rey_f^{-1}+\Rey_{N,f}^{-1})$.

Within a K41-type similarity theory, the model has an energy cascade that the relaxation term truncates at a micro-scale that depends on the global velocity $U$, the global length $L$, the deconvolution parameters $\alpha$ and $N$, and the filter radius $\delta$.  The cascade is truncated only if $\chi$ exceeds the threshold $\chi_{\min}\simeq U(\delta/L)^{1/3}\alpha^{-1/3}$, and the micro-scale is never smaller than $\sqrt{\alpha/(3N+2)}\,\delta$.  The extra dissipation at the cut-off balances the transfer of energy from the power input and so prevents a non-physical accumulation of energy at the cut-off scale.  With the choice \eqref{eq:chi},
\[
  \chi\simeq\frac{U}{L^{1/3}}\,\delta^{1/3}\Big(1+\frac1\alpha\Big)^{N+1},
\]
the micro-scale is $\delta$, and the number of degrees of freedom per time step needed for a three-dimensional simulation with \eqref{eq:model} is
\begin{equation*}
  N_{\mathrm{dof}}\simeq\Big(\frac{L}{\delta}\Big)^3,\qquad\text{independent of }\Rey .
\end{equation*}
Counting also the number of time steps, which scales like $N_{\mathrm{dof}}^{1/3}$, the predicted speed-up over a direct numerical simulation is
\begin{equation*}
  \Big(\frac{N_{\mathrm{dof}}^{\mathrm{NSE}}}{N_{\mathrm{dof}}}\Big)^{4/3}\simeq\Big(\frac{\Rey^{9/4}}{L^3\delta^{-3}}\Big)^{4/3}=\Big(\frac{\delta}{L}\Big)^4\Rey^3 .
\end{equation*}
The value \eqref{eq:chi} is derived for fully developed turbulence.

Several questions are open.  (1) The similarity theory should be tested numerically, in particular the threshold $\chi_{\min}$, the lower limit $\eta_*$ of the micro-scale, and the trade-off between $\alpha$ and $N$ in Remark~\ref{rem:alpha}.  (2) The analysis uses periodic boundary conditions, for which all operators are Fourier multipliers; the Dirichlet case requires the compatibility conditions discussed in \cite{MZ18}.  (3) Scales below $\eta_-$ are not controlled by the relaxation term.  Whether energy reaches them, and how the viscous range connects to the relaxation-dominated band, is not addressed by the present theory.  (4) A nonlinear relaxation term, as in \cite{LN07, ELN07}, would give a damping rate that grows with the size of the fluctuations and might remove the threshold.

\bibliographystyle{siam}
\bibliography{trmitlar}

\end{document}